\documentclass[10pt]{amsart}
\usepackage{array,youngtab}
\usepackage{color, graphicx, comment}
\usepackage{tikz}
\usetikzlibrary{arrows}

\usepackage{mathtools}

\definecolor{darkgreen}{rgb}{0.,0.5,0.}

\allowdisplaybreaks

\numberwithin{equation}{section} \overfullrule 5pt
\newtheorem{Theorem}{Theorem}

\newtheorem{Conjecture}[Theorem]{Conjecture}
\newtheorem{Lemma}[Theorem]{Lemma}
\newtheorem{Proposition}[Theorem]{Proposition}

\theoremstyle{definition}

\DeclareMathOperator{\dif}{dif}
\DeclareMathOperator{\cycdif}{cycdif}

\DeclareMathOperator{\product}{prod}
\newcommand{\Sym}{\mathfrak{S}}

\title[On the existence of permutations]{%
	On the existence of permutations conditioned by 
certain rational functions} 
\date{October 6, 2019}

\author{Guo-Niu Han}
\address{I.R.M.A., UMR 7501, Universit\'e de Strasbourg
et CNRS, 7 rue Ren\'e Descartes, F-67084 Strasbourg, France}
\email{guoniu.han@unistra.fr}



\begin{document}
\begin{abstract} 
We prove several conjectures made by Z.-W. Sun on the existence of permutations conditioned by certain rational functions. 
Furthermore, we fully  characterize  all integer values 
of the ``inverse difference" rational function.
Our proofs consist of both investigation of the 
mathematical properties of the rational functions and brute-force attack
by computer for finding special permutations. 

\end{abstract}

\maketitle

\section{Introduction}

Permutations (see, for example, \cite{Flajolet2009Sed, Stanley2012EC1}) are studied in almost every branch of mathematics and also in computer science.
The number of permutations $\pi=(\pi(1), \pi(2), \ldots, \pi(n))\in \Sym_n$
of $\{1,2,\ldots,n\}$ is $n!$. In \cite{Sun2018P} Z.-W. Sun made several
conjectures about the existence of permutations conditioned by certain rational functions. In the paper we confirm three of them by proving the following theorem.

\begin{Theorem}\label{th:main1}
	{\rm(i)} For any integer $n>5$, there is a permutation $\pi\in \Sym_n$ 
	such that
	\begin{equation}\label{eq:dif}
		\sum_{k=1}^{n-1} \frac{1}{\pi(k)-\pi(k+1)}=0.
\end{equation}
	{\rm (ii)} For any integer $n>7$, there is a permutation $\pi\in \Sym_n$ 
	such that
	\begin{equation}\label{eq:cdif}
		\sum_{k=1}^{n-1} \frac{1}{\pi(k)-\pi(k+1)} + \frac{1}{\pi(n)-\pi(1)}=0.
	\end{equation}
	{\rm (iii)} For any integer $n>5$, there is a permutation $\pi\in \Sym_n$ 
	such that
	\begin{equation}\label{eq:prod}
		\sum_{k=1}^{n-1} \frac{1}{\pi(k)\pi(k+1)}=1.
\end{equation}
\end{Theorem}

\medskip

Since $n!$ is a huge number for large $n$,
the generation of all $n!$ permutations of $n$ by computer is already a challenge \cite{Sedgewick1977}. 
For this reason \eqref{eq:dif}-\eqref{eq:prod} have only been 
verified for very small $n$.
Our proof of Theorem \ref{th:main1} consists of both investigation of the 
mathematical properties of the rational functions and brute-force attack
by computer for finding certain special permutations.

\medskip

Furthermore, we can fully  characterize  all integer values 
of the ``inverse difference" rational function given 
in the left-hand side of \eqref{eq:dif}.
We define
\begin{equation}\label{eq:dif_full}
V_n=\left\{\sum_{k=1}^{n-1} \frac{1}{\pi(k)-\pi(k+1)}  : \pi\in\Sym_n\right\}.
\end{equation}
Since $a\in V_n$ implies $-a\in V_n$ by the reverse of permutation, we only need
to study the nonnegative integer values of $V_n$. For example, $n=5$, 
the value set $V_5$ contains the following nonnegative rational number:
$$
\frac {1}{12}, \frac {1}{6}, \frac{1}{4}, 
\frac {1}{3}, \frac 12, \frac {7}{12}, \frac 23, \frac 34, 1, \frac 76,
\frac 43, \frac 32, \frac {19}{12}, \frac 74, \frac {11}{6}, \frac {23}{12}, 
2, \frac {13}{6}, \frac {11}{4}, 4.
$$
We see that there are three integers $1,2,4$ in the above list.

\begin{Theorem}\label{th:main2}
	We have 
	$V_3 \cap \mathbb{N} = \{2\}$,
	$V_5 \cap \mathbb{N} = \{1,2,4\}$, and for $n\not= 3, 5$,
	\begin{equation}
		V_n \cap \mathbb{N} = \{0\leq j\leq n-1 \mid j\not= n-2\}. 
\end{equation}
\end{Theorem}

The proofs of Theorems \ref{th:main1} and \ref{th:main2} will be given in Section 2.
Notice that we are still not able to prove three other conjectures of Sun.
Let us reproduce them below for interested readers. 

\begin{Conjecture}
	(i) \cite[Conj. 4.7(ii)]{Sun2018P}
For any integer $n>6$, there is a permutation $\pi\in \Sym_n$ 
	such that
	\begin{equation}\label{eq:sum}
		\sum_{k=1}^{n-1} \frac{1}{\pi(k)+\pi(k+1)}=1.
\end{equation}
Also, for any integer $n>7$, there is a permutation $\pi\in \Sym_n$ 
	such that
	\begin{equation}\label{eq:cycsum}
		\sum_{k=1}^{n-1} \frac{1}{\pi(k)+\pi(k+1)} + \frac{1}{\pi(n)+\pi(1)}=1.
	\end{equation}
	(ii) \cite[Conj. 4.8(ii)]{Sun2018P}
For any integer $n>7$, there is a permutation $\pi\in \Sym_n$ 
	such that
	\begin{equation}\label{eq:sqdif}
		\sum_{k=1}^{n-1} \frac{1}{\pi(k)^2-\pi(k+1)^2}=0.
\end{equation}

\end{Conjecture}

Motivated by \eqref{eq:sqdif}, we make the following conjecture.
\begin{Conjecture}\label{conj:cycsqdif}
For any integer $n>11$, there is a permutation $\pi\in \Sym_n$ 
	such that
	\begin{equation}\label{eq:cycsqdif}
		\sum_{k=1}^{n-1} \frac{1}{\pi(k)^2-\pi(k+1)^2}
		+ \frac{1}{\pi(n)^2-\pi(1)^2}=0.
\end{equation}
\end{Conjecture}
Conjecture \ref{conj:cycsqdif} has been checked for $11<n<28$ by computer. We list 
below the permutations satisfying  \eqref{eq:cycsqdif}, which are found by our computer program 
in a highly non-trivial way.
{\small
\begin{align*}
\pi_{12} &= ( 1,4,3,5,7,2,12,8,10,11,9,6 ),\\
\pi_{13} &= ( 1,2,12,8,9,6,11,10,7,5,13,4,3 ),\\
\pi_{14} &= ( 1,2,12,9,6,4,3,13,8,7,5,10,14,11 ),\\
\pi_{15} &= ( 1,9,2,3,12,10,11,5,4,14,6,15,13,8,7 ),\\
\pi_{16} &= ( 1,3,2,4,5,11,16,14,10,8,6,12,9,15,13,7 ),\\
\pi_{17} &= ( 1,3,2,4,5,9,15,6,12,16,11,10,14,13,8,7,17 ),\\
\pi_{18} &= ( 1,3,2,4,6,5,7,13,8,14,12,16,10,18,17,9,11,15 ),\\
\pi_{19} &= ( 1,3,2,4,6,5,7,8,12,18,17,13,9,15,11,10,16,19,14 ),\\
\pi_{20} &= ( 1,3,2,4,6,5,7,18,8,13,12,17,9,20,16,19,10,11,15,14 ),\\
\pi_{21} &= ( 1,3,2,4,6,5,7,17,8,20,16,9,12,18,15,13,19,21,11,14,10 ),\\
\pi_{22} &= ( 1,3,2,4,6,5,7,8,20,13,17,22,18,12,9,15,21,19,16,11,10,14 ),\\
\pi_{23} &= ( 1,3,2,4,6,14,10,18,12,8,20,7,5,21,15,11,17,13,22,23,16,19,9 ),\\
\pi_{24} &= ( 1,3,2,4,6,14,10,18,12,8,5,9,21,11,24,16,20,22,17,15,13,19,23,7 ),\\
\pi_{25} &= ( 1,3,2,4,6,14,10,18,12,8,5,16,24,9,21,23,7,17,15,11,13,22,20,19,25 ),\\
\pi_{26} &= ( 1,3,2,4,6,14,10,18,12,8,22,13,5,23,16,20,19,21,9,7,17,11,25,15,24,26 ),\\
\pi_{27} &= ( 1,3,2,4,6,14,10,18,12,8,22,13,9,5,11,21,23,16,26,19,25,27,17,15,24,20,7 ).
\end{align*}
}

\section{Proofs}

Let $\Phi_{\dif}(\pi)$, $\Phi_{\cycdif}(\pi)$,  and 
$\Phi_{\product}(\pi)$ denote the three rational functions expressed in  the left-hand side of \eqref{eq:dif}, \eqref{eq:cdif},  and \eqref{eq:prod},respectively. 
The following {\it Link lemma} is useful for our construction. 

\begin{Lemma}[Link]\label{th:link}
	Let $\sigma\in\Sym_s $ and $\tau\in\Sym_t$ be two permutations 
	on $\{1,2,\ldots, s\}$ and $\{1,2,\ldots, t\}$, respectively, such that $\sigma(s)=s, \tau(1)=1$ and $\Phi_{\dif}(\sigma)=\Phi_{\dif}(\tau)=0$.
	We define the {\it ``link"} of the two permutations
	 $\rho \in \Sym_{s+t-1}$ by
	\begin{equation*}
		\rho(k)=
		\begin{cases}
			\sigma(k),  & \text{if $1\leq k\leq s$}\\
			s-1+\tau(k-s+1). & \text{if $s+1\leq k\leq s+t-1$} 
		\end{cases}
	\end{equation*}
	Then, we have  $\Phi_{\dif}(\rho)=0$.
	Furthermore, if $\tau(t)=t$, we have
$\rho(s+t-1)=s+t-1$.
\end{Lemma}
\begin{proof}
	Notice that in the definition of $\rho$, if we allow $k=s$ in the second
	case, the expression will give the same definition of $\rho(s)$ as in the first case, 
	since $\sigma(s)=s=s-1+\tau(1)$. Hence,
\begin{align*}
	\Phi_{\dif}(\rho)&=
	\sum_{k=1}^{s-1} \frac{1}{\rho(k)-\rho(k+1)}+
	\sum_{k=s}^{s+t-2} \frac{1}{\rho(k)-\rho(k+1)}\\
	&=\sum_{k=1}^{s-1} \frac{1}{\sigma(k)-\sigma(k+1)}+
	\sum_{k=1}^{t-1} \frac{1}{\tau(k)-\tau(k+1)}\\
	&=\Phi_{\dif}(\sigma) + \Phi_{\dif}(\tau)\\
	&=0.
\end{align*}
	Furthermore, if $\tau(t)=t$, it is easy to see that $\rho(s+t-1)=s+t-1$.
	\end{proof}

Let us write the link $\rho$ of $\sigma$ and $\tau$ by $\langle \sigma, \tau\rangle$.
\smallskip

{\it Example}. Take $\sigma=(1,4,2,5,3,6)$ and $\tau=(1,3,2,4)$. We verify that
\begin{align*}
	\Phi_{\dif}(\sigma) &= -\frac{1}{3} + \frac 12 - \frac 13 + \frac 12 -\frac 13 =0, \cr
	\Phi_{\dif}(\tau) &= -\frac 12 + \frac 11 -\frac 12=0.  \cr
\end{align*}
We have $\rho = \langle \sigma, \tau \rangle = (1,4,2,5,3,6,8,7,9)$ and
\begin{equation*}
	\Phi_{\dif}(\rho) = \left(-\frac{1}{3} + \frac 12 - \frac 13 + \frac 12 -\frac 13\right) 
	+ \left(-\frac 12 + \frac 11 -\frac 12\right)=0. 
\end{equation*}

If $\tau(t)=t$, since the link $\rho = \langle \sigma, \tau \rangle  $ also
satisfies the conditions $\rho(s+t-1)=s+t-1$ 
and $\Phi(\rho)=0$, we can ``link" again and obtain
$
\langle \rho, \tau \rangle =
\langle \langle \sigma, \tau \rangle, \tau \rangle .
$

\medskip

The following proposition is a slightly stronger version of  Theorem \ref{th:main1}(i) for the rational function $\Phi_{\dif}$. 

\begin{Proposition}\label{th:prop1}
	For any integer $n>5$, there is a permutation $\pi\in \Sym_n$ 
	such that $\pi(1)=1, \pi(n)=n$, and $\Phi_{\dif}(\pi)=0$. \\
\end{Proposition}

\vspace{-6mm}

\begin{proof}
Let
	\begin{align*}
		\sigma_0 &= (1,4,2,5,3,6), \\
		\sigma_1 &= (1,3,2,4), \\
		\sigma_2 &= (1,3,6,4,7,5,2,8),
	\end{align*}
	and $\tau=\sigma_1=(1,3,2,4)$. We have $\Phi_{\dif}(\sigma_j)=0$ for $j=0,1,2$, and $\tau(1)=1, \tau(4)=4$.
By repeated application of the link algorithm, we obtain the following three families of permutations:
\begin{align*}
		&(1,4,2,5,3,6), \\
		&(1,4,2,5,3,6\mid8,7,9), \\
		&(1,4,2,5,3,6\mid8,7,9\mid11,10,12), \\
		&\vdots\\
		&(1,3,2,4) ,\\
		&(1,3,2,4\mid6,5,7) ,\\
		&(1,3,2,4\mid6,5,7\mid9,8,10) ,\\
		&\vdots\\
		&(1,3,6,4,7,5,2,8) ,\\
		&(1,3,6,4,7,5,2,8\mid10,9,11) ,\\
		&(1,3,6,4,7,5,2,8\mid10,9,11\mid13,12,14) ,\\
		&\vdots
\end{align*}
	\vspace{-3mm}
	of length
\begin{align*}
	n&=6,9,12,15,\ldots    \qquad (3k) \\
	n&=4,7,10,13,\ldots    \qquad (3k+1) \\
	n&=8,11,14,17\ldots    \qquad (3k+2)
\end{align*}
	Hence we have constructed one permutation $\pi\in\Sym_n$ for each $n>5$ such that $\Phi_{\dif}(\pi)=0$ and $\pi(1)=1$ and $\pi(n)=n$.
\end{proof}

The following proposition is another enhanced version of  Theorem \ref{th:main1}(i) for $\Phi_{\dif}$, 
which is crucial for proving Theorem \ref{th:main1}(ii) for $\Phi_{\cycdif}$. 
The difference between 
Propositions \ref{th:prop1} and \ref{th:prop2} lies in the value of $\pi(n)$.

\begin{Proposition}\label{th:prop2}
	For any integer $n>7$, there is a permutation $\pi\in \Sym_n$ 
	such that $\pi(1)=1, \pi(n)=n-1$ and $\Phi_{\dif}(\pi)=0$. \\
\end{Proposition}

\vspace{-7mm}

\begin{proof}
	Let
\begin{align*}
	\alpha_8&= (1,2,4,8,6,5,3,7),\\
	\alpha_9&= (1,4,2,5,9,3,7,6,8),\\
	\alpha_{10}&=(1,2,6,3,7,8,5,4,10,9),\\
	\alpha_{11}&=(1,2,3,4,6,5,9,8,7,11,10),\\
	\alpha_{12}&=(1,2,3,6,4,8,12,10,9,7,5,11).
\end{align*}
	We have $\alpha_j(1)=1$, $\alpha(j)=j-1$, and $\Phi_{\cycdif}(\alpha_j)=0$ for $j=8,9,\ldots, 12$. The proposition is true for $j=8,9, \ldots, 12$. For $n\geq 13$ and $k=n-7\geq 6$, take the permutation 
	$\sigma\in\Sym_{k}$ obtained in Proposition \ref{th:prop1}, i.e.,
	$\sigma(1)=1, \sigma(k)=k, \Phi_{\dif}(\sigma)=0$.
	Then, the link $\rho=\langle \sigma, \alpha_8\rangle$  satisfies 
	$\rho(1)=1, \rho(n)=n-1$ and $\Phi_{\dif}(\rho)=0$.
	For example, for $n=13$ and  $k=6$,
$$
	\rho=\langle (1,4,2,5,3,6), (1,2,4,8,6,5,3,7)\rangle = (1,4,2,5,3,6,7,9,13,11,10,8,12).
$$
Hence, the proposition is true for any $n>7$.
\end{proof}

Now we are ready to prove part (ii) of Theorem \ref{th:main1}
for the rational function $\Phi_{\cycdif}$.

\begin{proof}[Proof of Theorem \ref{th:main1}(ii)]
	If $n=2k$ is even, we can easily check that
	the permutation
\begin{equation*}
		\pi=(1,2,3,\ldots, k-1, k, 2k, 2k-1, \ldots, k+3, k+2, k+1),
\end{equation*}
which is obtained by the concatenation of the increasing permutation
	of $\Sym_k$ and the decreasing permutation of $\{j \mid k+1\leq j\leq 2k\}$,
	satisfies 
	\begin{equation*}
\Phi_{\cycdif}(\pi)=
		\left(-\frac 11 - \frac 11 - \cdots -\frac 11 -\frac {1}{k}
		+\frac 11 + \frac 11 + \cdots +\frac 11\right) +\frac {1}{k}
=0.
\end{equation*}
The odd case is more complicated. First, we define 
\begin{align*}
	\beta_9 &= (2,1,4,5,9,3,7,6,8), \\
	\beta_{11} &= (1,2,11,5,4,8,7,9,3,6,10), \\
	\beta_{13} &= (1,2,13,3,5,4,9,8,10,6,11,7,12).
\end{align*}
	We have $\Phi_{\cycdif}(\beta_j)=0$ for $j=9,11,13$. Next, for $n=2k+1\geq 15$, i.e., $k\geq 7$ and $m=k+1\geq 8$, by Propositions \ref{th:prop1} 
	and \ref{th:prop2}, there exist two permutations
	$\sigma\in\Sym_{k}$ and $\sigma\in\Sym_{m}$ such that

	{\rm (i)} $\sigma(1)=1$, $\sigma(k)=k$, and $\Phi_{\dif}(\sigma)=0$;

	{\rm (ii)} $\tau(1)=1$, $\tau(m)=m-1$, and $\Phi_{\dif}(\tau)=0$.

	Let $\tau'$ be the permutation of $\{j\mid k+1\leq j \leq 2k+1\}$
	obtained by adding $k$ in the reverse of $\tau$:
	\begin{equation*}
		\tau'=(k+\tau(m), k+\tau(m-1), k+\ldots, k+\tau(3), k+\tau(2), k+\tau(1) ).
\end{equation*}
	Notice that the first and last elements of $\tau'$ are $k+\tau(m)=2k$ and 
	$k+\tau(1)=k+1$, respectively. Let $\rho$ be the concatenation of $\sigma$ and $\tau'$. 
	We can verity that  $\Phi_{\dif}(\tau') = -\Phi_{\dif}(\tau)=0$ and
	\begin{equation*}
		\Phi_{\cycdif}(\rho) = \Phi_{\dif}(\sigma) + \frac {1}{k-2k} + \Phi_{\dif}(\tau') + \frac {1}{(k+1)-1}=0.
	\end{equation*}
For example, for $n=15$, $k=7$ and $m=8$, we have
$\sigma=(1,3,2,4,6,5,7)$ and $\tau= (1,2,4,8,6,5,3,7)$, so that
	$\tau'=(14, 10, 12, 13, 15, 11, 9, 8 )$. Our final permutation $\rho$
	is the concatenation of $\sigma$ and $\tau'$:
	\begin{equation*}
		\rho=(1,3,2,4,6,5,7, \ 14, 10, 12, 13, 15, 11, 9, 8).
	\end{equation*}
	We can check that $\Phi_{\cycdif}(\rho)=0$.
\end{proof}

To prove part (iii) of Theorem \ref{th:main1}
concerning the rational function $\Phi_{\product}$, we need the following lemma. Also, it is much more convenient to describe the construction in the increasing binary trees model \cite{Stanley2012EC1, Flajolet2009Sed}.

\begin{Lemma}[Insertion]\label{th:insertion}
	Let $\sigma=\sigma(1)\sigma(2)\cdots\sigma(n-1)\in\Sym_{n-1}$ 
	be a permutation and  $\tau\in\Sym_n$ be the permutation
	obtained by insertion of the letter $n$ into $\sigma$:
	$$\tau=\sigma(1)\cdots \sigma(j)\,n\,\sigma(j+1)\cdots\sigma(n-1). 
	\qquad (j=1,2,\ldots, n-2)$$
	Then, $\Phi_{\product}(\sigma)=\Phi_{\product}(\tau)$ if and only if
	$\sigma(j)+\sigma(j+1)=n$.
\end{Lemma}
\begin{proof} We have
\begin{align*}
	\Phi_{\product}(\sigma) = \sum_{k=1}^{n-2} 
	\frac {1}{\sigma(k)\sigma(k+1)} 
	&= \cdots + \frac {1}{\sigma(j)\sigma(j+1)} +\cdots \cr
	\Phi_{\product}(\tau) = \sum_{k=1}^{n-1} 
	\frac {1}{\tau(k)\tau(k+1)} 
	&= \cdots + \frac {1}{\sigma(j)\,n} + \frac {1}{n\,\sigma(j+1)} +\cdots \cr
	&	= \cdots + \frac {\sigma(j)+\sigma(j+1)}{n\, \sigma(j)\sigma(j+1)} +\cdots 
\end{align*}
	Hence, $\Phi_{\product}(\sigma)=\Phi_{\product}(\tau)$ if and only if
 $$\frac {1}{\sigma(j)\sigma(j+1)} 
=	 \frac {\sigma(j)+\sigma(j+1)}{n\, \sigma(j)\sigma(j+1)},$$
i.e.,  $\sigma(j)+\sigma(j+1)=n$.
\end{proof}

\tikzset{
  treenode/.style = {align=center, inner sep=0pt, text centered,
    },
  arnn/.style = {treenode, circle,  draw=black, text width=1.3em },
  arnr/.style = {treenode, circle,  draw=red, text width=1.3em },
  arnb/.style = {treenode, circle,  draw=blue, text width=1.3em },
 }

	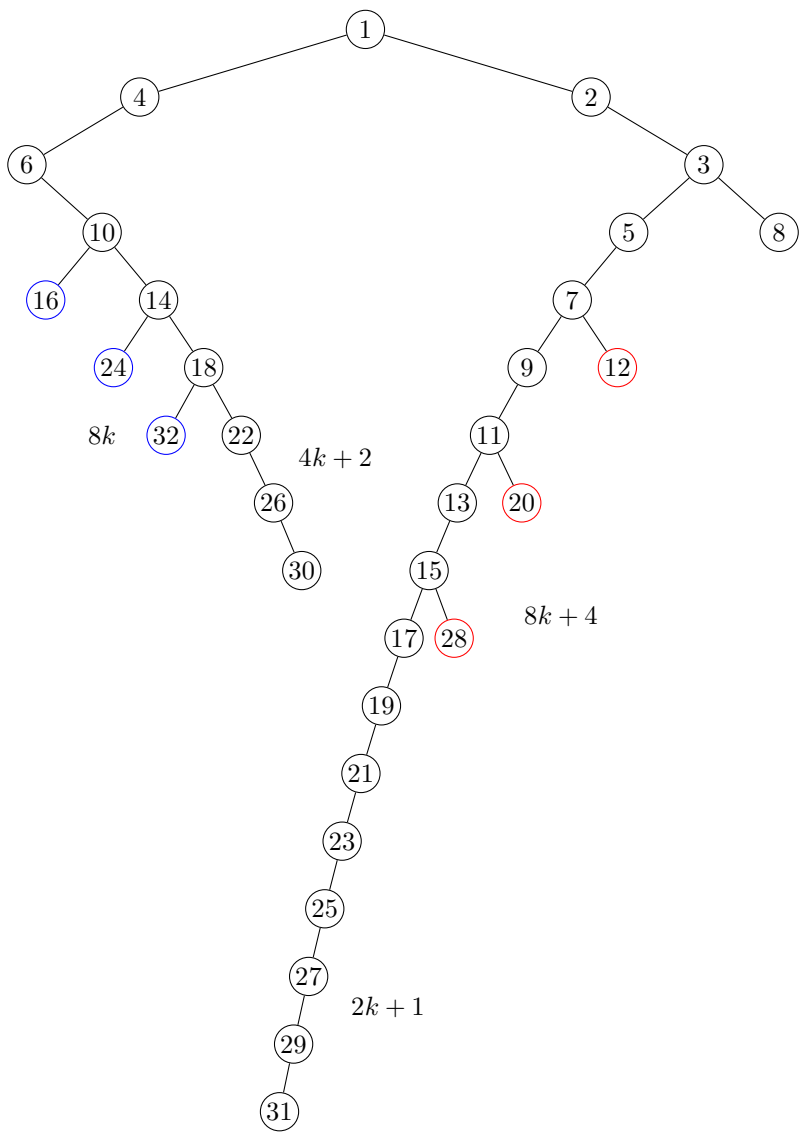
\begin{figure*}[tbp]
\begin{tikzpicture}[-,>=stealth',level/.style={sibling distance = 6cm/#1,
  level distance = 0.9cm}] 
\node [arnn] {1}
child{ node [arnn] {4} 
	child{ node [arnn] {6} 
		child[missing]{  {}}
		child{ node [arnn] {10} 
		child{node[arnb]  {16}}
		child{node[arnn]  {14}
			child{node[arnb]  {24}}
			child{node[arnn]  {18}
			child{node [arnb]{32}}
			child{node[arnn]  {22}
			child[missing]{}
			child{node[arnn]  {26}
			child[missing]{}
			child{node[arnn]  {30}
			}
}
	}
	}
		}
		} 
	}
	child[missing]{	}                            
}
child{ node [arnn] {2}
				child[missing]{ }
        child{ node [arnn] {3}
					child{ node [arnn] {5}
						child{ node [arnn] {7}
					child{ node [arnn] {9}
						child{ node [arnn] {11}
child{ node [arnn] {13}
child{ node [arnn] {15}
child{ node [arnn] {17}
child{ node [arnn] {19}
child{ node [arnn] {21}
child{ node [arnn] {23}
child{ node [arnn] {25}
child{ node [arnn] {27}
child{ node [arnn] {29}
child{ node [arnn] {31}
}
child[missing]{}
}
child[missing]{}
}
child[missing]{}
}
child[missing]{}
}
child[missing]{}
}child[missing]{}
}child[missing]{}
	}child{node [arnr]{28}}
}
child[missing]{}
}
					child{ node [arnr] {20}}
						}
						child[missing]{}
					}
					child{ node [arnr] {12}}
						}
						child[missing]{}
					}
					child{ node [arnn] {8}}
        }
}
; 
\node at (-3.5,-5.4) {$8k$};
\node at (-0.4,-5.7) {$4k+2$};
\node  at (2.6,-7.8) {$8k+4$};
\node  at (0.3,-13) {$2k+1$};
\end{tikzpicture}
\caption{The increasing binary tree for $\delta_{32}$}
\label{fig:1}
\end{figure*}

Now we use the insertion lemma to prove part (iii) of our main theorem.

\begin{proof}[Proof of Theorem \ref{th:main1}(iii)]
	Let
\begin{align*}
\delta_6 &=  (2, 1, 3, 4, 5, 6), \\
	\delta_7&= (2, 1, 3, 7, 4, 5, 6),\\
	\delta_8&=(6,4,1,2,7,5,3,8).
\end{align*}
We verify that $\Phi_{\product}(\delta_j)=1$ for $j=6,7,8$.
	By insertion lemma, we define $\sigma_9$
	by inserting the letter $9$ between $2$ and $7$ in $\sigma_8$:
	$$
	\sigma_9=(6,4,1,2,{\bf 9}, 7,5,3,8).
	$$
	Next, we insert $10$ in $\sigma_9$ between $6$ and $4$:
	$$
	\sigma_{10}=(6,{\bf 10}, 4,1,2, 9, 7,5,3,8).
	$$
	For the insertion of $11$, we have two possible positions, namely, between 
	$(2, 9)$  and $(3, 8)$. We choose the position $(2,9)$ and define
	$$
	\sigma_{11}=(6, 10, 4,1,2,{\bf 11}, 9, 7,5,3,8).
	$$
The crucial idea is to show that this kind of insertion can be repeatedly applied  as many times as we want, starting from $\sigma_8$. Thus, we obtain the desired permutation $\sigma_n\in\Sym_n$ for each $n\geq 8$.
To understand the general pattern, we take a rather big example
	with $n=32$. Our permutation $\delta_{32}$ is as follows:
	\begin{multline*}
		\delta_{32}=(6, 16, 10, 24, 14, 32, 18, 22, 26, 30, 4, 1, 2, \\
31, 29, 27, 25, 23, 21, 19, 17, 15, 28, 13, 11, 20, 9, 7, 12, 5, 3, 8),
	\end{multline*}
	which can be represented by the increasing binary tree \cite{Stanley2012EC1, Flajolet2009Sed} in Figure~\ref{fig:1}.
We see that the nodes of the form $2k+1, 4k+2, 8k, 8k+4$ all appear in the tree structure. Hence,
the insertion can be repeatedly applied to reach each $\delta_{n}$
for $n\geq 8$.
\end{proof}

As proved in Theorem \ref{th:main1}, 
for any integer $n>5$, there is a permutation $\pi\in \Sym_n$ 
such that $\Phi_{\dif} (\pi) =0$. 
For a permutation $\pi$, the value of $\Phi_{\dif}(\pi)$ is, a priori,
a rational number.
Theorem \ref{th:main2}
 provides a full 
characterization of the integer values of the rational function $\Phi_{\dif}$.

\begin{proof}[Proof of Theorem \ref{th:main2}]
	In fact, we need to prove a stronger statement by replacing each $V_n$ in 
	the theorem by
\begin{equation}
	V_n'=\left\{\sum_{k=1}^{n-1} \frac{1}{\pi(k)-\pi(k+1)}  : 
	\pi\in\Sym_n, \pi(n)=n\right\}.
\end{equation}

It is easy to see that $n-1\in V_n'$ by taking the identity permutation 
	$(1,2,\ldots n)$. Also, the maximal value of $\Phi_{\dif}(\pi)$ is $n-1$, 
	so that $m\not\in V_n$ for $m\geq n$. 
We can check by computer for all $n\leq 6$. 
If $n\geq 7$, we prove by induction on $n$.
	First, $0\in V_n'$ by Proposition \ref{th:prop1}.
	Next, for $1\leq m \leq n-3$, by the induction hypothesis, there is a permutation $\tau\in\Sym_{n-1}$ such that $\tau(n-1)=n-1$ and $\Phi_{\dif}(\tau)=m-1$. 
	We define
	$$
	\pi=(\tau(1), \tau(2), \ldots, \tau(n-1), n).
	$$
	We verify that $\Phi_{\dif}(\pi)= 1 + \Phi_{\dif}(\tau)=m$.
Finally, for each $\sigma\in\Sym_n$ which is not the identity permutation, 
	we see that there exists at least one negative term in the summation \eqref{eq:dif}.
	Hence $\Phi_{\dif}(\sigma)<n-2$, and $n-2\not\in V_n$.
\end{proof}






\end{document}